\documentclass[12pt]{article}
\usepackage[english]{babel}
\usepackage[utf8]{inputenc}

\usepackage{mathtools}
\usepackage{comment}
\usepackage{xcolor}
\usepackage{tikz}

\usepackage{relsize}
\usetikzlibrary{positioning}
\usepackage{graphicx}
\usepackage{amsmath}
\usepackage{amssymb}
\usepackage{amsthm}
\usepackage{amsfonts}

\usepackage{setspace}
\usepackage{subcaption}

\usepackage{hyperref}
\hypersetup{
	colorlinks = true,
	linkcolor = {blue},
	citecolor={black}
}
\usepackage{smartref}
\usepackage{cleveref}

\newcommand{\Q}{\mathbb{Q}}           
\newcommand{\R}{\mathbb{R}}            
\newcommand{\C}{\mathbb{C}}           
\newcommand{\Qalg}{\overline{\Q}}     

\newcommand{\T}{\top}     
\newcommand{\G}{\mathcal{G}}          
\newcommand{\II}{\mathcal{I}}          
\DeclareMathOperator{\Ker}{Ker}    

\newtheorem{prop}{Proposition}
\newtheorem{teo}{Theorem}

\newtheorem{conj}{Conjecture}
\newtheorem*{conj*}{Conjecture}
\newtheorem{defin}{Definition}

\title{Vertex distinction with subgraph centrality: a proof of Estrada's conjecture and some generalizations}
\author{Francesco Ballini\footnotemark[2]\hspace{0.5cm} Nikita Deniskin\footnotemark[2]}
\date{18 July 2020}

\begin{document}
\maketitle

\footnotetext[1]{Scuola Normale Superiore, Pisa, Italy. E--mail : {\tt francesco.ballini@sns.it, nikita.deniskin@sns.it}.}

\begin{abstract}

Centrality measures are used in network science to identify the most important vertices for transmission of information and dynamics on a graph.
One of these measures, introduced by Estrada and collaborators, is the $\beta$-subgraph centrality, which is based on the exponential of the matrix
$\beta A$, where $A$ is the adjacency matrix of the graph and $\beta$ is a real parameter (``inverse temperature"). We prove that for algebraic
$\beta$, two vertices with equal $\beta$-subgraph centrality are necessarily cospectral.
We further show that two such vertices must have the same degree and eigenvector centralities. Our results settle 
a conjecture of Estrada and a generalization of it due to Kloster, Kr\'al and Sullivan. We also discuss possible extensions of our results.

\end{abstract}


\section{Introduction}

Centrality measures have been used to determine the importance of a vertex in a graph, with many applications in biology, finance, sociology, epidemiology, and more generally in network science. Among many such measures, we focus here on  subgraph centrality, which is based on counting the number of closed walks of different lengths passing through each node. This measure has been 
successfully used in the study of protein-protein interaction networks, in the analysis of traffic and other transportation networks, and in several studies of brain networks, to name just a few applications; see, for instance, \cite{BKcentralitymeasures, Estrada, EPHwalkentropies, EHstatisticalmechanics,EHcommunicability, EHB12, ERVsubgraphdefinition}.

Let $\G=(V,E)$ be a simple undirected graph with $|V|=n$ vertices and adjacency matrix $A$. Later we will also consider the case where $\G$ is a directed and weighted graph.
The $\beta$-subgraph centrality with $\beta \geq 0$  is defined as $[e^{\beta A}]_{ii}$ for  each vertex $i$ of $\G$. 
It was introduced by Estrada and Rodr\'iguez-Vel\'azquez in \cite{ERVsubgraphdefinition} for $\beta=1$, 
as a node centrality measure. Two years later, Estrada and Hatano \cite{EHstatisticalmechanics} introduced a generalization of it involving the tuneable parameter $\beta$.
The idea is to write $e^{\beta A}$  as a power series expansion:

\begin{equation}
\begin{aligned}
e^{\beta A}\, =\;& I+ \beta A+ \frac{\beta^2}{2}A^2 + \frac{\beta^3}{3!} A^3+\cdots,\\
[e^{\beta A}]_{ii}\, = \;& 1 + \beta[A]_{ii} + \frac{\beta^2}{2} [A^2]_{ii} + \frac{\beta^3}{3!} [A^3]_{ii}+\cdots.
\end{aligned}
\label{espansione_expA}
\end{equation}

As we have anticipated, the $\beta$-subgraph centrality of node $i$ is then given by $[e^{\beta A}]_{ii}$. Hence, the $\beta$-subgraph centrality is strictly related to the number of closed walks starting from $i$, since the number of such walks of length $r$ is $[A^r]_{ii}$. By weighting walks of length $r$ by $\beta^r/r!$, longer closed walks are penalized. Nodes that are visited by
many short, closed walks are considered important. The role of the parameter $\beta$, known as the
``inverse temperature," is that of giving more or less weight to walks
of a given length, and also to model situations where the network is subject to some external ``stress". 

In the above-mentioned applications in network science, it is often useful to determine when two ``different" vertices have the same centrality measure. As a first step, one can ask which graphs have all vertices with equal subgraph centrality. Highly symmetrical graphs, such as vertex-transitive graphs, satisfy this condition; a wider class of graphs that satisfy is that of walk-regular graphs. It was conjectured that there are not any others:

\begin{conj*}Given $\beta>0$, a graph $\G$ has all vertices with the same $\beta$-subgraph centrality if and only if $\G$ is walk-regular.
\end{conj*}

In this paper we will show that this conjecture is true for all algebraic $\beta$, by proving a stronger result: we will show that if $\beta$ is algebraic, two vertices $i,j$ have the same $\beta$-subgraph centrality if and only if they are cospectral. This implies the conjecture because a graph is walk-regular if and only if all its vertices are cospectral. 

\vspace{0.2cm} 

In Section 2, we give all the necessary definitions and show various  formulations of the conjecture. In Section 3 we introduce the Lindemann-Weierstrass Theorem. In Section 4 we prove the main result (Theorem \ref{main_thm}) and Theorem \ref{prop_main}, which is the key element for its proof. In Section 5 we discuss some generalizations of our results, and possible further developments.

\section{Preliminaries}

It is convenient to introduce the following terminology. 
\begin{defin}$\G$ is \textit{$\beta$-subgraph regular} if $\forall \; i,j\in V$, $[e^{\beta A}]_{ii} = [e^{\beta A}]_{jj}$.
\end{defin}

In \cite{ERVsubgraphdefinition} examples were given of graphs with vertices with equal degree, eigenvector, closeness and betweenness centralities, but different 1-subgraph centralities. This led to the following conjecture:

\begin{conj}[Estrada, Rodr\'iguez-Vel\'azquez \cite{ERVsubgraphdefinition}]
	Let $\G$ be a $1$-subgraph regular graph. Then the degree, closeness, eigenvector, and betweenness centralities are also identical for all nodes.
	\label{conjec1}
\end{conj}

Some counterexamples for the closeness and betweenness centralities were found independently by Rombach and Porter \cite{RPdiscrimination} and by Stevanovi\'c \cite{Stevanovic}, but the conjecture remained open for degree and eigenvector centralities.
\vspace{0.2cm}

We recall the following definition:
\begin{defin} $\G$ is \textit{walk-regular} if $\forall \; i,j\in V$  and for every positive integer $r$, $[A^r]_{ii} = [A^r]_{jj}$.
\end{defin}

From the power series expansion of eq.~(\ref{espansione_expA}), it follows immediately that a walk-regular graph is also $\beta$-subgraph regular for all $\beta$. From here on, we assume that $\beta \ne 0$ to avoid trivialities.

A related quantity is the  \textit{walk entropy} of a graph \cite{EPHwalkentropies,EHstatisticalmechanics}, defined as
$$S(\G,\beta) = - \sum_{i=1}^n p_i \ln p_i, \quad p_i=
\frac{[e^{\beta A}]_{ii}}{\text{Tr} [e^{\beta A}]}\,.$$
It is easy to see that the walk entropy is maximized (and equal to $\ln n$) if and only if the graph $\G$ is $\beta$-subgraph regular.  
In \cite{EPHwalkentropies} it was conjectured that $\G$ is walk-regular if and only if $\G$ is $\beta$-subgraph regular for all $\beta > 0$. This was proved true by Benzi in the following stronger form:

\begin{teo}[Benzi \cite{Benzi}, Theorem 2.2]
	A graph $\G$ is walk-regular if and only if $\G$ is $\beta$-subgraph regular for all $\beta \in I \subseteq \R$, where $I$ is any set of real numbers containing an accumulation point.
\end{teo}

In the same paper, it was conjectured that if a graph $\G$ is $\beta$-subgraph regular for only one value of $\beta$, then it is necessarily walk-regular (also in \cite{Estrada} this was conjectured for the special case $\beta=1$). The general case was shown to be false by Sullivan et al. in \cite{HKSwalkregularity,KKScounterexample}, 
by exhibiting a (infinite) family of non walk-regular graphs (also, non degree-regular), for each of which there exists a value of $\beta$ such that the graph is $\beta$-subgraph regular; incidentally, this counterexample also falsified an incorrect proof of the above-mentioned
conjecture given in \cite{EdP}.
Nevertheless, for any non walk-regular graph $\G$ there can be only finitely many values of $\beta$ that make $\G$ $\beta$-subgraph regular.

In  \cite{KKScounterexample}, the following conjecture was put forth:

\begin{conj}[Kloster, Kr\'al, Sullivan \cite{KKScounterexample}, Conjecture 5]
	A graph $\G$ is walk-regular if and only if there exists a rational $\beta  > 0$ such that $\G$ is $\beta$-subgraph regular.
	\label{conjec5}
\end{conj}

We will show that this conjecture is true in a stronger form, by requiring $\beta$ only to be any algebraic number. Also, our result applies not just to undirected graphs, but also to directed graphs with diagonalizable adjacency matrices. 

We recall that in the case of a directed graph the interpretation of $[A^r]_{ii}$ in terms of closed walks remains valid, provided
that a ``closed walk" is understood as a directed walk starting and ending at the same vertex. 

\vspace{0.3cm}

For $\G$ either a directed or an undirected graph, we introduce the following terminology.

\begin{defin}Two vertices $i,j$ of $\G$ are \textit{cospectral} if for every integer $r >0$, $[A^r]_{ii} = [A^r]_{jj}$.
\end{defin}

Observe that by the Hamilton-Cayley Theorem, it is sufficient to check $n-1$ values of $r$ to determine cospectrality. If there exists an automorphism $\varphi$ of $\G$ such that $\varphi(i)=j$, then $i,j$ are cospectral; however, there are examples of cospectral vertices which are not  related by an automorphism. One of such examples can be found in \cite{Schwenk}, which was the first to make use of cospectral vertices, although without defining them explicitly. See \cite{GSstronglycospectral} for many other equivalent conditions for two vertices to be cospectral.  

\begin{defin} Two vertices $i,j$ of $\G$ are \textit{$\beta$-subgraph equivalent} if $[e^{\beta A}]_{ii} = [e^{\beta A}]_{jj}$.
\end{defin}

From the Taylor series expansion it is clear that if $i,j$ are cospectral, then they are $\beta$-subgraph equivalent for all $\beta$. We will show that for $\G$ an undirected graph, or a directed graph with diagonalizable adjacency matrix, if $\beta$ is an algebraic number and $i,j$ are $\beta$-subgraph equivalent, then they are cospectral.

This implies a proof of Conjecture \ref{conjec5}, because a graph is walk-regular if and only if all its vertices are cospectral, and it is $\beta$-subgraph regular if and only if all its vertices are $\beta$-subgraph equivalent.

\section{Algebraic numbers and the Lindemann-Weierstrass Theorem}

We recall that $a\in \C$ is an algebraic number if there exists a nonzero polynomial $p(x)\in \Q[x]$ such that $p(a)=0$. The set of all algebraic numbers is a field and it will be denoted by $\Qalg$.

\begin{prop}Let $B$ be a $n \times n$ matrix with all its entries $B_{ij} \in \Qalg$. Let $\Ker(B)\subseteq \C^n$ be the null-space of $B$ with $\dim \Ker(B)=d \geq 1$. Then there exists a basis $\{v_1,\ldots , v_d\}$ of $\Ker(B)$ such that all the entries of each vector are algebraic numbers.
	\label{prop_kernel}
\end{prop}

\begin{proof}We can see $B$ as a matrix over the field $\Qalg$. Gaussian elimination holds 
	on every ground field, so we can apply it to the rows of $B$ and find a basis of the null-space $\{v_1,\ldots , v_d\}$, with $v_i \in \Qalg^n$.
	Since $\Qalg \subseteq \C$, we have that $\{v_1,\ldots , v_d\}$ is also a basis for $\Ker(B)$ when viewed as a complex-valued vector space.
\end{proof}

\begin{prop}Let $B$ be a $n\times n$ matrix with all its entries $B_{ij} \in \Qalg$. If $B$ is non-singular, then the inverse matrix $B^{-1}$ has all its entries in $\Qalg$.
	\label{prop_inverse}
\end{prop}

\begin{proof} The inverse of $B$ can be computed explicitly: $[B^{-1}]_{ij} = (-1)^{i+j}\,  \frac{\det(B\setminus(j,i))}{\det(B)} $, where $\det(B\setminus(j,i))$ is the minor of the matrix obtained removing row $j$ and column $i$. It is clear that 
	$\det(B)$ and $\det(B\setminus(j,i))$ are both algebraic numbers, so $[B^{-1}]_{ij} \in \Qalg$ for every $i,j$.
	
\end{proof}

We now introduce the Lindemann-Weierstrass Theorem, which is the central ingredient for the main result. The theorem, proven in 1885 combining the work of Hermite, Lindemann and Weierstrass, is a milestone of Transcendental Number Theory. We state it here in a formulation due to Baker \cite{Baker}.

\begin{teo}[Lindemann-Weierstrass]
	Let $a_1, \ldots, a_n$ be distinct algebraic numbers. Then the exponentials $e^{a_1},\ldots, e^{a_n}$ are linearly independent over the algebraic numbers. In other words, for every choice of $c_1,\ldots c_n \in \overline{\Q}$, not necessarily distinct, we have:
	\begin{equation}
	c_1 e^{a_1} \,+ \,\cdots \,+\, c_n e^{a_n}\,=\,0 \; \; \; \iff \; \;\; c_i = 0 \;\; \;\;\forall\; \;1\leq i \leq n. \end{equation}
\end{teo}

\begin{proof}
	See for instance \cite{Siegel} for a proof with an historical perspective, \cite{Baker} for a simpler argument or \cite{Nathanson} for a proof using Padé approximants.
\end{proof}

Notice that the result has many important consequences: i.e., the transcendence of $e$ (choosing $a_1=1$ and $a_2=0$) and the transcendence of $\pi$ (choosing $a_1=i \pi$ and $a_2=0$). For the history of the Lindemann-Weierstrass Theorem, see \cite{Brezinski}.

\section{Main result}

Let $\G$ be a graph with adjacency matrix $A$, and assume that $A$ is diagonalizable, say $A=Q D Q^{-1}$. We are mostly interested in undirected graphs, where the latter property is always true (because $A$ is real symmetric); nonetheless, we can extend the result at least to directed graphs with diagonalizable adjacency matrix.
\vspace{0.3cm}

Let $q_{ij}$ be the $(i,j)^{th}$ entry of $Q$ and $\widehat{q}_{ij}$ that of $Q^{-1}$. Let $(\lambda_1,\ldots, \lambda_n)=\text{diag}(D)$ be the (possibly non-distinct) eigenvalues of $A$. We will use $(\mu_1,\ldots ,\mu_d)$ with $d\leq n$ to denote the eigenvalues without repetition, with $\mu_i$ of multiplicity $m_i$. Up to permutation, we can assume $\lambda_1=\cdots=\lambda_{m_1}=\mu_1$, $\lambda_{m_1+1}=\cdots=\lambda_{m_1+m_2}=\mu_2$, and so on. For ease of notation, let $\II_h=\{ k \,|\, \lambda_k =\mu_h\}$ be the set of all indices of the multiple occurrences of eigenvalue $\mu_h$.

\vspace{0.3cm}

From $A^r = Q D^r Q^{-1}$ and $e^{\beta A} = Q \,e^{\beta D} \, Q^{-1}$ we can group equal eigenvalues together to obtain:

\begin{equation}
\begin{aligned}
[A^r]_{ii} = &\sum\limits_{k=1}^n \,q_{ik}\;\widehat{q}_{ki} \, \lambda_k^r = \left(\,\sum\limits_{k\in\II_1} q_{ik}\;\widehat{q}_{ki} \right)\mu_1^r + \cdots + \left(\,\sum\limits_{k\in\II_d} q_{ik}\;\widehat{q}_{ki} \right)\mu_d^r=\\
=&\;C_{1\,i} \;\mu_1^r \,+\, C_{2\,i}\; \mu_2^r\,+\,\cdots\,+\,C_{d\,i} \;\mu_d^r\,;
\end{aligned}
\label{Ar_autoval}
\end{equation}

\begin{equation}[e^{\beta A}]_{ii} = \sum\limits_{k=1}^n \,q_{ik}\;\widehat{q}_{ki}  \;e^{\beta\lambda_k} = C_{1\,i}\;e^{\beta\mu_1} \,+\,C_{2\,i}\;e^{\beta\mu_2} \,+\, \cdots \,+\, C_{d\,i}\;e^{\beta\mu_d}\,
\label{expA_autoval}
\end{equation}

where $C_{h\,i}= \sum\limits_{k\in\II_h} q_{ik}\;\widehat{q}_{ki}$. The next proposition is the key argument in the proof of the main result.

\begin{teo}
	Let $\G$ be a directed graph with adjacency matrix $A$, and assume that $A$ is diagonalizable. Let $\beta\neq 0$ be an algebraic number. If two vertices $i,j$ are $\beta$-subgraph equivalent, then they are cospectral.
	\label{prop_main}
\end{teo}

\begin{proof}
	We would like to apply the Lindemann-Weierstrass Theorem, so we need to prove that (with the above notation) the exponents  $\beta \mu_h$ and the coefficients $C_{h\,i} = \sum\limits_{k\in\II_h} q_{ik}\;\widehat{q}_{ki} $ are all algebraic numbers.
	\vspace{0.3cm}
	
	The entries of $A$ are either 0 or 1, so the characteristic polynomial $P_A(x)=\det(xI-A)$ has integer coefficients. The roots of $P_A(x)$ are $\mu_1,\ldots ,\mu_d$, so  they all are algebraic numbers. Since $\beta\in\Qalg$, then also $\beta\mu_1,\ldots, \beta\mu_d$ are algebraic numbers.
	\vspace{0.3cm}
	
	Observe that the coefficients $C_{h\,i}$ in equations (\ref{Ar_autoval}), (\ref{expA_autoval}) can be obtained for many possible choices of $Q$, as long as $A=QDQ^{-1}$ holds. We will construct an appropriate $Q$ with algebraic numbers in all entries.
	\vspace{0.25cm}
	
	For every eigenvalue $\mu_h$, let $B=A-\mu_h I$. Using Proposition \ref{prop_kernel} we can find vectors $\{v_1,\ldots ,v_{m_h}\}$ which form a basis of $\Ker(B)$ and such that all their components are in $\Qalg$.  Hence, we can 
	construct a matrix $Q$ which has the $m_h$ columns relative to the eigenvalue $\mu_h$
	equal to the above-defined vectors $\{v_1,\ldots ,v_{m_h}\}$. 
	$Q$ has all the entries in $\Qalg$, and so does its inverse $Q^{-1}$ by Proposition \ref{prop_inverse}. This implies that $\forall \, h,i$, the coefficients $C_{h\,i}=\sum\limits_{k\in\II_h} q_{ik}\;\widehat{q}_{ki} $ are algebraic numbers.
	
	
	We can now prove the result. The hypothesis is $[e^{\beta A}]_{ii} = [e^{\beta A}]_{jj}$ which we can write as in equation (\ref{expA_autoval}) as:
	
	\begin{equation*}
	\begin{aligned}
	[e^{\beta A}]_{ii} = \sum\limits_{k=1}^n \,q_{ik}\;\widehat{q}_{ki}  \;e^{\beta\lambda_k} &= \,C_{1\,i}\;e^{\beta\mu_1}\, +\,C_{2\,i}\;e^{\beta\mu_2} \,+ \,\cdots \,+\, C_{d\,i}\;e^{\beta\mu_d},\\
	[e^{\beta A}]_{jj} = \sum\limits_{k=1}^n \,q_{jk}\;\widehat{q}_{kj}  \,e^{\beta\lambda_k} &=\, C_{1\,j}\;e^{\beta\mu_1} \,+\,C_{2\,j}\;e^{\beta\mu_2}\, +\, \cdots\, +\, C_{d\,j}\;e^{\beta\mu_d},
	\end{aligned}
	\end{equation*}
	
	\begin{equation*}
	0=[e^{\beta A}]_{ii}-[e^{\beta A}]_{jj} =\, (C_{1\,i}-C_{1\,j})\;e^{\beta\mu_1} \,+\, \cdots\, +\, (C_{d\,i}-C_{d\,j})\;e^{\beta\mu_d}.    
	\end{equation*}
	
	Since for every $h$ the exponents $\beta\mu_h$ and coefficients $C_{h\,i},\, C_{h\,j}$ are algebraic numbers, and also $\beta \mu_h$ are all distinct because $\beta \neq 0$ and $\mu_h$ are pairwise distinct, we can apply the Lindemann-Weierstrass Theorem to obtain that $C_{h\,i}=C_{h\,j}\;\; \forall \, 1\leq h\leq d$.
	From this it follows that for all positive integers $r$,
	\begin{equation}
	\begin{aligned}
	[A^r]_{ii} =&\;C_{1\,i} \;\mu_1^r + C_{2\,i}\; \mu_2^r+\cdots+C_{d\,i} \;\mu_d^r=\\
	=&\;C_{1\,j} \;\mu_1^r + C_{2\,j}\; \mu_2^r+\cdots+C_{d\,j} \;\mu_d^r=[A^r]_{jj}\,,
	\end{aligned}
	\end{equation}
	which means that $i,j$ are cospectral in $\G$. The proof is complete.
	
\end{proof}

\textbf{Remark.} Observe that if $\G$ is an undirected graph, then its adjacency matrix $A$ is symmetric and therefore it is diagonalizable; hence the result of Theorem \ref{prop_main} can be applied.

\vspace{0.3cm}

We will now prove the Conjectures \ref{conjec1} and \ref{conjec5} stated in Section 2:

\begin{teo}[Main Result] 
	Let $\beta>0$ be an algebraic number and let $\G$ be a connected undirected graph with adjacency matrix $A$.
	\begin{enumerate} 
		\item 	$\G$ is $\beta$-subgraph regular if and only if $\G$ is walk-regular.
		\item If two vertices $i$, $j$ are  $\beta$-subgraph equivalent, then the degree and eigenvector centralities of $i$ and $j$ are equal.
		\item If $\G$ is $\beta$-subgraph regular, then the degree and eigenvector centralities are also identical for all nodes.
	\end{enumerate}
	\label{main_thm}
\end{teo}
\begin{proof}
	\textbf{(1)} If $\G$ is walk-regular, then by the Taylor series expansion of $[e^{\beta A}]_{ii}$ it follows that $\G$ is $\beta$-subgraph regular for every $\beta\in \R$. 
	
	If $\G$ is $\beta$-subgraph regular for $\beta \in \Qalg$, this means that $\forall \,i,j $ we have $[e^{\beta A}]_{ii} = [e^{\beta A}]_{jj}$. By Theorem \ref{prop_main}, we have that $[A^r]_{ii} = [A^r]_{jj}$ for every $r>0$ and for every $i,j$, which is the definition of walk-regularity.\\
	
	\textbf{(2)} 	The degree centrality of $i$ is the number of edges incident in $i$, which is $[A^2]_{ii}$. Since Theorem \ref{prop_main} implies that $[A^r]_{ii} = [A^r]_{jj}$ for every integer $r>0$, it follows that $[A^2]_{ii}=[A^2]_{jj}$.
	\vspace{0.2cm}

	Let us take $Q$ as in the proof of Theorem \ref{prop_main}. Since $A$ is real symmetric, $Q$ can be transformed into an orthogonal matrix still satisfying $A=QDQ^{-1}$ by applying Gram-Schmidt orthogonalization and column normalization; these operations preserve the algebraicity of its entries. So $Q^{-1}=Q^\T$ and the coefficients are simply $C_{h\,i}=\sum\limits_{k\in\II_h} q_{ik}\;\widehat{q}_{ki} = \sum\limits_{k\in\II_h} q_{ik}^2$.

	Up to permutation, we can assume that $\lambda_1$ is the eigenvalue with the greatest absolute value. Since $\G$ is undirected and connected, by the Perron-Frobenius Theorem \cite{perronfrobenius}, $\lambda_1$ is a simple eigenvalue with a non-negative eigenvector $(q_{1\,1},\ldots ,q_{n\,1})^\T$. The eigenvector centrality of vertex $i$ is defined as $q_{i\,1}$.
	\vspace{0.2cm}
	
	In the proof of Theorem \ref{prop_main} we have obtained that for $i,j$ which are $\beta$-subgraph equivalent, $C_{1\,i}=C_{1\,j}$. Since $\lambda_1$ is a simple eigenvalue, $C_{1\,i}=q_{i\,1}^2$ and $C_{1\,j}=q_{j\,1}^2$. We conclude that $q_{i\,1}=q_{j\,1}$ because they are both non-negative, proving that $i$ and $j$ have the same eigenvector centrality.
	
	\vspace{0.2cm}
	\textbf{(3)} It follows from point 2 and the fact in a $\beta$-subgraph regular graph all vertices are $\beta$-subgraph equivalent.
	
\end{proof}

\textbf{Remark.} Point 1 implies that the value(s) of $\beta$ in the counterexample found in \cite{KKScounterexample} is necessarily a transcendental number.

\section{Generalizations and remarks}

For any sufficiently regular function $f$ (analytic and with radius of convergence in $0$ greater than $\rho(A)$) the matrix function $f(A)$ can be calculated using the Taylor series expansion. 
Defining the \textit{diagonal entry function} as $f_D(i)=[f(A)]_{ii}$, it is possible to obtain properties of the graph and of the vertices $i,j$ by comparing $f_D(i)$ and $f_D(j)$. The subgraph centrality is a special case obtained by taking $f(x)=e^{\beta x}$. Other functions have also been studied in literature, for example $f(x)=\frac{1}{1-\alpha x}$ (with $0<\alpha < \frac{1}{\rho(A)}$) which gives the resolvent subgraph centrality; see, for instance, \cite{EHnetwork}.

If two vertices $i,j$ are cospectral, then by power series expansion it follows that  $f_D(i)=[f(A)]_{ii}=[f(A)]_{jj}=f_D(j)$: this means that $i$ and $j$ cannot be distinguished by any diagonal entry function. However, the function $f(x)=e^{\beta x}$ with algebraic $\beta$ has the ``maximum resolution" among all diagonal entry functions: by Theorem \ref{prop_main}, two non cospectral vertices must have different $\beta$-subgraph centralities.

It is not yet known a ``simple" function which can always distinguish vertices up to graph automorphism. Nevertheless, the subgraph centrality can distinguish non cospectral vertices, and that is the limit for any diagonal entry function. \\

We can observe that the proof of Theorem \ref{prop_main} does not need $A$ to be the adjacency matrix of a graph, but only that the roots of $P_A(x)$ and the eigenvectors of $A$ are algebraic. This is true if all the entries of $A$ are rational (or even algebraic) numbers. 

\begin{prop}
	Let $A\in \Qalg^{n\times n}$ be a diagonalizable matrix. If for $1\leq i,j\leq n$ and $\beta \in \Qalg$ we have that $[e^{\beta A}]_{ii}= [e^{\beta A}]_{jj}$, then for every integer $r>0$ we have $[A^r]_{ii}= [A^r]_{jj}$.
\end{prop}

We can see $A$ as the adjacency matrix of a weighted directed graph, with algebraic weights (possibly negative).
\\

The next question is whether the result can be generalized to a non-diagonalizable matrix $A$ (both for $A$ adjacency matrix of a directed graph, or more generally for  any $A$ with algebraic coefficients).

We have been able to obtain a partial answer to this question.

\begin{prop}
	Let $A\in \Qalg^{n\times n}$ with Jordan normal form $J$, i.e. $A=QJQ^{-1}$. Assume that $\lambda_1$ has index $\leq 2$ (its largest Jordan block has size $\leq 2$) and all other eigenvalues have index 1.
	If $[e^{\beta A}]_{ii}= [e^{\beta A}]_{jj}$, then  for every integer $ r>0$ we have $[A^r]_{ii}= [A^r]_{jj}$.
\end{prop}

\begin{proof}
	The Jordan normal form of $A=QJQ^{-1}$ is the following, with $m$ copies of the block $J_1$:
	
	$$J=Q^{-1}AQ=\begin{pmatrix} 
	J_1 &        &     &  \\
	& \ddots &     & \\
	&        & J_1 & \\
	&        &     & \lambda_{2m+1} &\\
	&        &     &           & \ddots \\
	& & & & & \lambda_n
	\end{pmatrix} , \hspace{1cm} 
	J_1=\begin{pmatrix} 
	\lambda_1 & 1 \\
	0& \lambda_1
	\end{pmatrix}.$$
	
	To calculate $A^r$ and $e^{\beta A}$, we need $J^r$ and $e^{\beta J}$, which are block-diagonal with the blocks relative to $J_1$ equal to:
	
	$$J_1^r=\begin{pmatrix} 
	\lambda_1^r & r \lambda_1^{r-1} \\
	0& \lambda_1^r
	\end{pmatrix} \hspace{1cm}
	e^{\beta J_1} = \begin{pmatrix}
	e^{\beta \lambda_1} & \beta e^{\beta \lambda_1}\\
	0& e^{\beta \lambda_1}
	\end{pmatrix}$$
	
	We obtain thus:
	
	$$[A^r]_{ii} = \sum\limits_{k=1}^n \,q_{ik} \;\widehat{q}_{ki}\; \lambda_k^r + \sum\limits_{l=1}^{m}  \,q_{i\,2l-1} \;\,\widehat{q}_{2l\,i} \; r \lambda_1^{r-1},$$
	$$[e^{\beta A}]_{ii} = \sum\limits_{k=1}^n \,q_{ik} \;\widehat{q}_{ki} \;e^{\beta \lambda_k} + \sum\limits_{l=1}^{m}  \,q_{i\,2l-1}\,\; \widehat{q}_{2l\,i}\;\beta e^{\beta \lambda_1}.$$
	
	Setting in this case $C_{h\,i} = \sum\limits_{k\in \II_h} q_{ik} \;\widehat{q}_{ki}$, by Lindemann-Weierstrass Theorem we have that $C_{h\,i } = C_{h\,j}$ for all $h\geq 2$. By looking at the coefficient of $e^{\beta \lambda_1}$ we obtain: 
	$$C_{1\,i} + \sum\limits_{l=1}^{m}\,  q_{i\,2l-1} \;\widehat{q}_{2l\,i}\,\beta = C_{1\,j} + \sum\limits_{l=1}^{m} \; q_{j\,2l-1} \;\widehat{q}_{2l\,j}\,\beta$$
	
	Using the relation $I=QQ^{-1}$, we have $1=I_{ii} = \sum\limits_{k=1}^n\,q_{ik}\, \widehat{q}_{ki}=\sum\limits_{h=1}^n C_{h\,i}$. Using the same relation for $I_{jj}$, we obtain that $C_{1\,i} = C_{1\,j}$, and so $\sum\limits_{l=1}^{m}  \,q_{i\,2l-1}\; \widehat{q}_{2l\,i} = \sum\limits_{l=1}^{m} \, q_{j\,2l-1}\; \widehat{q}_{2l\,j}$. From this it follows that $[A^r]_{ii}=[A^r]_{jj}$ for all $r>0$, as desired.
	
\end{proof}

We believe that the result is true for all non-diagonalizable matrices, so we set forth the following conjecture:

\begin{conj}
	Let $A$ be the adjacency matrix of a directed, unweighted graph, with $A$ non-diagonalizable. If for two vertices $i,j$ and for $\beta\in \Qalg$ we have $[e^{\beta A}]_{ii} = [e^{\beta A}]_{jj}$, then for every integer $r>0$ we have $[A^r]_{ii} = [A^r]_{jj}$.
\end{conj}

Considering the application of Lindemann-Weierstrass Theorem, it is quite possible that this conjecture holds also for all non-diagonalizable matrices $A \in \Qalg^{n\times n}$. 

\section{Acknowledgements}
The authors would like to thank Michele Benzi for bringing this problem to their attention, for useful discussions on the topic and for his very helpful review of the manuscript, and an anonymous referee for helpful comments and suggestions.

\end{document}